%% file: A-kernelshift-revison-2.tex
\input{newcommand.tex}
\usepackage[margin=90pt]{geometry}
\usepackage{verbatim}
\usepackage{mathdots}
\usepackage{scalerel,stackengine}
\usepackage[colorlinks=true,linkcolor=blue,urlcolor=blue]{hyperref}
\usepackage{color}
\hypersetup{
	colorlinks   = true,
	citecolor    = teal,
	linkcolor    = purple,
	pagebackref
}

\definecolor{labelkey}{rgb}{1,0,0}
\usepackage{array}
\usepackage{soul}
\usepackage[normalem]{ulem}

\newcommand{\mff}{\mfrak{f}}

\newcommand{\red}{\textcolor{black}}

\newcommand{\cy}{\textcolor{black}}
\nc{\ff}{\mathbf{f}}
\nc{\bx}{\bold{x}}
\nc{\by}{\bold{y}}
\nc{\bz}{\bold{z}}
\nc{\ba}{\bold{a}}
\nc{\phm}{\Phi_{[m]}}
\nc{\lan}{\langle}
\nc{\Fm}{F_{[m]}}
\nc{\ran}{\rangle}
\nc{\Fd}{F^\dagger}
\nc{\Rd}{R^\dagger}
\nc{\mlow}{m_{\mathrm{l}}}
\nc{\mup}{m_{\mathrm{u}}}
\nc{\ord}{\mb{ord }}
\nc{\bXp}{\bX_{\mathrm{prim}}}
\nc{\bPsi}{\bold{\Psi}}
\nc{\mult}{\mathrm{mult}}
\nc{\mbB}{\mathbbm{B}}
\nc{\mfor}[1]{{#1}^{\mathrm{for}}}
\nc{\latder}{\tilde{\partial}}
\nc{\bg}{{\bold h}}
\nc{\bff}{{\bold f}}
\nc{\bu}{{\bold u}}
\nc{\mO}{\mathcal{O}}
\nc{\tTheta}{\tilde{\Theta}}
\nc{\io}{\iota}
\nc{\cb}{\mathbf{b}}
\nc{\tF}{\tilde{F}}
\nc{\mg}{\mathfrak{g}}
\nc{\bc}{\mathbf{c}}
\nc{\Witt}{\text{$\mathbb{W}$itt}}
\nc{\mW}{\mathbb{W}}
\nc{\mH}{\mathbf{H}}
\nc{\Frac}{\text{Frac}}

\nc{\bHnew}{\tilde{\bH}}
\nc{\bGa}{\bar{\mathbb{G}}_a}
\nc{\Iso}{\mathrm{FIso}}

\nc{\cblue}[1]{\color{black} {#1} \color{black}}
\nc{\Wm}[1]{W_{[m]#1}}
\nc{\Ima}[1]{I_{[m]{#1}}}
\nc{\pr}{\mathrm{pr}}
\nc{\fra}{\mathfrak{f}}

\usepackage [english]{babel}
\usepackage [autostyle, english = american]{csquotes}
\MakeOuterQuote{"}

\title{Kernels of Arithmetic Jet Spaces and Frobenius Morphisms}
\author{Rajat Kumar Mishra and Arnab Saha}
\date{}
\email{rajat@math.tifr.res.in, arnab.saha@iitgn.ac.in}
\address{Tata Institute of Fundamental Research, Mumbai, Maharashtra 400005, India} 
\address{Indian Institute of Technology Gandhinagar, Gujarat 382355, India}

\usepackage{cite}
\begin{document}
	
	\begin{abstract}
		For any $\pi$-formal group scheme $G$, 
		the Frobenius morphism between arithmetic jet spaces 
		restricts to generalized kernels of the projection map. Using the functorial
		properties of such kernels of arithmetic jet spaces, we show that this morphism
		is indeed induced by a natural ring map between the shifted $\pi$-typical Witt 
		vectors. 
		
		In the special case when $G = \hG$, the arithmetic jet space, as well
		as the generalized kernels, is affine $\pi$-formal planes with Witt 
		vector addition as the group law. In that case the above morphism is the 
		multiplication by $\pi$ map on Witt vector schemes. In fact, the system
		of arithmetic jet spaces and generalized kernels of any $\pi$-formal group
		scheme $G$ together with their morphisms and identities satisfied among them  
		is a generalization of the case of the Witt vector scheme with 
its maps such as the Frobenius, Verschiebung and multiplication by $\pi$.

		\color{black}
	\end{abstract}
	\maketitle
	\section{Introduction}
	

Let us fix a Dedekind domain $\mO$ of characteristic $0$ with finite residue fields 
	and a nonzero prime ideal $\mf{p}$ of  $\mO$. Let $k$ be the residue field of $\mf{p}$ of cardinality $q$ which is  a power of a prime $p$, and let $\pi$ be a uniformizer of the discrete valuation ring $\mO_{\mf{p}}$. Let $R$ be a flat
	$\mO_{\mf{p}}$-algebra that is complete with respect to the $\pi$-adic topology with a fixed $\pi$-derivation $\delta$ on it. 

The $\pi$-typical Witt vector functor of length $(n+1)$ where $n \in \{0,1,\dots\}$, 
denoted $\WW_n$, is a 
functor from the category of $R$-algebras to itself where for every $R$-algebra
$B$, the ring structure of $\WW_n(B)$ is uniquely
determined such that the Witt polynomials induce a natural transformation
between $\WW_n$ and the $(n+1)$-product ring functor. The system of functors
$\{\WW_n\}_{n=0}^\infty$ admits the following important operators:
\begin{enumerate}
\item The Frobenius map $F: \WW_{n+1} \map \WW_n$ for all $n \geq0$
\item The Verschiebung map $V:\WW_n \map \WW_{n+1}$
\item The multiplication by $\pi$ map, $\pi : \WW_n \map \WW_n$.
\end{enumerate}
The above operators satisfy the following fundamental identities:
\begin{align}
\label{introidentity}
F \circ V & = \pi \nonumber \\
F^{\circ 2} \circ V &= F \circ V \circ F. 
\end{align}
Let $G$ be a $\pi$-formal group scheme over $\Spf R$. The $n$-th arithmetic
jet space $J^nG$ introduced by Buium is naturally a $\pi$-formal group 
scheme over $\Spf R$.
There is a natural projection map $u:
J^nG \map G$ of $\pi$-formal group schemes and let $N^nG:= \ker (u)$
be the kernel of $u$ with $\iota : N^nG \map J^nG$  the natural closed 
immersion. 	

The arithmetic jet spaces admit a canonical lift of Frobenius 
morphism $\phi: J^{n+1}G \map J^nG$.
By \cite{PaSa}, one has a natural isomorphism of functors $N^nG \simeq
J^{n-1}(N^1G)$ for all $n \geq 1$ and hence the associated 
Frobenius of 
arithmetic jet spaces yields $\fra: N^nG \map N^{n-1}G$. 
Then the above morphisms satisfy
the canonical identity
\begin{align}
\label{phiifra}
\phi^{\circ 2} \circ \iota = \phi \circ \iota \circ \fra.
\end{align}
In the case when $G = \hG$, 
by the functor of points definition of $n$-th arithmetic jet spaces, one
sees directly that $\WW_n \simeq J^n(\hG)$ where $\hG$ is the $\pi$-formal
affine line $\hA^1$ under addition. It follows that the kernel $N^1(\hG) 
\simeq \hG$ and therefore from the isomorphism $N^n(\hG) \simeq J^{n-1}(N^1G)$
we obtain $N^n(\hG) \simeq \WW_{n-1}$ and hence $\fra:N^n(\hG) \map 
N^{n-1}(\hG)$ coincides with the Frobenius map $F:\WW_{n-1} \map \WW_{n-2}$.
It is also evident that the inclusion map $\iota: N^n(\hG) \map J^n(\hG)$ is 
nothing but the Verschiebung map $V: \WW_{n-1} \map \WW_n$ 
under the above isomorphism of functors. 

Therefore, we observe that the above identity \eqref{phiifra} recovers 
the second equation in \eqref{introidentity} in the particular case when
the group scheme $G = \hG$. A natural question is what the appropriate
 generalization would be of the multiplication by $\pi$ 
map, $\pi: \WW_n \map \WW_n$ in the case of arbitrary group schemes $G$ such
that when specialized to the case $G = \hG$, we recover the first identity
of \eqref{introidentity}.

To do so, one needs to consider a generalized notion of kernels of 
arithmetic jet spaces $J^nG$ as introduced in \cite{PaSa}. 
In this article, we will do so and construct canonical morphisms 
$\Phi_{[m]}$ between such systems of generalized kernels that will answer the 
question raised above. Consequently, the morphisms $F$, $V$ and 
$\pi$ of the $\pi$-typical Witt vectors are special instances of a much 
more general phenomenon for any arbirary group scheme $G$.

\subsection{Notations and geometric set up}
We will now explain our results in more detail.
Consider any $\pi$-formal scheme $X$ defined over $S=\Spf R$. In analogy with differential algebra, for each integer $n \ge 0$, the $n$-th arithmetic jet functor is defined as
	$$J^nX(B):=X(W_n(B))$$ where $W_n(B)$ denotes the $\pi$-typical Witt vectors of length $n+1$ associated to any $\pi$-adically complete 
	$R$-algebra $B$ \cite{Bo1,Drn,Lslhlt,Joyal}. By \cite{BePrSa,Bo2}, the functor $J^nX$ is representable by a $\pi$-formal scheme over $S$ which we also denote by $J^nX$. 
	The system of $\pi$-formal schemes $J^*X = \{J^nX\}_{n=0}^\infty$ is 
	called the canonical prolongation sequence.
	In the category of $\pi$-formal schemes, the above definition of $J^nX$ 
	coincides with the arithmetic jet space constructed of Buium \cite{Bu2}. 
	
	In \cite{Bu2}, Buium initiated the theory of delta characters for an abelian variety.  Subsequently, this theory was further developed in a series of articles, including \cite{Bar,BoGu,Bu3,Bu4,BuSa1,BuSa2,Hurlburt}. It has also found remarkable applications in diophantine geometry as in \cite{Bu1,BuPu,DP}. 

	Suppose $G$ is a $\pi$-formal group scheme over $S$. The restriction map
	on the $\pi$-typical Witt vectors induces the projection map $u: J^{m+n}G 
	\map J^mG$ for all $m,n \geq 0$.
	We define the kernel of $u$ to be the generalized kernel $N^{[m]n}G$.
	On the other hand, the Frobenius on the $\pi$-typical Witt vectors gives 
	a Frobenius morphism $\phi: J^{m+n}G \map J^{m-1+n}G$ of $\pi$-formal schemes 
	for all $m,n \geq 0$. Hence restricting the Frobenius $\phi$ on $N^{[m]n}G$
	induces a morphism $\phm:N^{[m]n}G \map N^{[m-1]n}G$.
	
	In this article, using the functorial description of $N^{[m]n}G$, we show
	that the morphism $\phm$ can be obtained from a natural ring map (denoted 
	$E_{[m]}$) between the  
	shifted $\pi$-typical Witt vectors. 
	
	When $G$ is a smooth commutative $\pi$-formal group scheme of relative 
	dimension $g$ over $S$ and the ramification index satisfies 
	$v_\pi(p) \leq p-2$, then
	by \cite{PaSa} we have an isomorphism of $\pi$-formal group schemes
	$N^{[m]n}G \simeq \WW_{n-1}^g$ where $\WW_{n-1}$ is the affine
$\pi$-formal 
	$n$-plane $\hA^n$ endowed with the group structure of the additive Witt vectors
	of length $n$. In this particular case, our morphism $\phm$ is in fact 
	the multiplication by $\pi$ map on $\WW_{n-1}^g$. However a more general 
	result is true which is explained in Theorem \ref{phmintro}.
	
	\red{Let $B$ be a $\pi$-adically complete $R$-algebra with structure map $R \map 
		B$.
		This induces a natural ring map $W_m(R) \map W_m(B)$. For any $m,n \geq 0$, 
		consider the truncation ring map between Witt vectors 
$W_{m+n}(B) \map W_m(B)$. Following \cite{PaSa}, we define the ring
		of {\it $m$-shifted Witt vectors of length $m+n+1$} as the fibre product of
		rings $W_{[m]n}(B):= W_m(R) \times_{W_m(B)} W_{m+n}(B)$ and denote
		$I_{[m]n}:W_{[m]n}(B)\to W_{m+n}(B)$ to be the natural map.} 
	
	In \cite{PaSa}, it was shown that the ring $W_{[m]n}(B)$ admits a canonical delta structure and the lift of 
	Frobenius associated to this delta structure is called the \textit{lateral Frobenius} ${F}_{[m]}:W_{[m]n}(B)\rightarrow W_{[m]n-1}(B^\phi)$ 
	where $B^\phi$ is the ring $B$ with the structure map given as the
composition $R 
	\stk{\phi}{\map} R \map B$. The lateral Frobenius
	satisfies the following  identity $$F^{m+2}\circ I_{[m]n}= F^{m+1}\circ I_{[m]n}\circ {F}_{[m]},$$ where $F:W_{m+n}(B)\rightarrow W_{m+n-1}(B^\phi)$ is the usual Frobenius map of Witt vectors.  Also let 
	$T: W_{[m]n}(B) \map W_{[m](n-1)}(B)$ denote the truncation map. 
	
	In this article we show that, for all $m\geq 1$, there exists a ring homomorphism $E_{[m]}:W_{[m]n}(B)\rightarrow W_{[m-1]n}(B^\phi)$ such that it commutes with the usual Frobenius map as follows
	\begin{align}
		\nonumber
		\xymatrix{
			\Wm{n}(B)\ar[d]_{E_{[m]}}  \ar[r]^{\Ima{n}} & W_{m+n}(B) \ar[d]^F\\
			W_{[m-1]n}(B^\phi) \ar[r]_{I_{[m-1]n}} & W_{m-1+n}(B^\phi).
		}
	\end{align}
	
\subsection{Statement of main results}
	Consider the system of $\pi$-formal group schemes $N^{[m]*}G :=
	\{N^{[m]n}G\}_{n=0}^\infty$. Then $N^{[m]*}G$ forms a prolongation sequence, 
	in fact $N^{[m]*}G  \simeq J^*(N^{[m]1}G)$ with $\fra_m: 
	N^{[m]n}G \map N^{[m](n-1)}G$ the associated lift of Frobenius morphism for all
	$n \geq 1$, which is induced from the lateral Frobenius $F_{[m]}$ described 
	above \cite{PaSa}.
Our first main result is the following, which is proved in greater generality in Theorem	\ref{phm&fm_comm}.
	
	\begin{theorem}
		For all $m$, the ring map $\Em$ induces a natural morphism of 
		prolongation sequences $\phm: N^{[m]*}G \map N^{[m-1]*}G$ such that 
		$\fra_m$ and $\phm$ commute with each other. In particular, the following 
		diagram commutes for every $m,n \geq 1$
		$$\xymatrix{
			N^{[m]n}G\ar[rr]^-{\mff_m}\ar[d]_-{\phm} & & N^{[m]n-1}G\ar[d]^-{\phm}\\
			N^{[m-1]n}G\ar[rr]^-{\mff_{m-1}} & & N^{[m-1]n-1}G.
		}$$
	\end{theorem}

Our second main result, proved in Theorem \ref{Psimiso}, is the following. 
	\begin{theorem}
		\label{phmintro}
Assume $v_\pi(p) \leq p-2$ and let $G$ be a smooth commutative $\pi$-formal scheme over $S$. 
For all $m$, the  natural 
		morphism of prolongation sequences $\phm :N^{[m]*}G \map N^{[m-1]*}G$ 
		satisfies the following commutative diagram
		\begin{align}
			\nonumber
			\xymatrix{
				0 \ar[r] & N^{[m]n}G \ar[dr]^{J^{n-1}(\Psi^{[m]}_1)} \ar[ddd]_-{\phm} \ar[rr]^{\iota_m} & & 
				J^{m+n}G \ar[ddd]_-\phi \ar[r]^{u} & J^mG \ar[r] \ar[ddd]_-\phi & 0\\
				& & (\bb{W}_{n-1})^g \ar[d]_-{\mult_\pi} & & & \\
				& & (\bb{W}_{n-1})^g & & & \\
				0 \ar[r] & N^{[m-1]n}G \ar[ur]_{J^{n-1}(\Psi_1^{[m-1]})} \ar[rr]_{\iota_{m-1}} & 
				&J^{m-1+n}G \ar[r]_{u} & J^{m-1}G \ar[r] & 0\\
			}
		\end{align}
		where $\Psi_1^{[m]}$ is as defined in Section \ref{KJS}.
	\end{theorem}
	We now illustrate our result in the special case when $G = \hG$, the 
	$\pi$-formal additive group scheme. 
	Note that for 
	all $n \geq 0$ we have $J^nG \simeq \bb{W}_n$ where $\bb{W}_n$ is the 
$\pi$-formal affine
	plane of relative dimension $n+1$ over $S$ endowed with the group law of the 
	addition of the Witt vectors. Hence we have the following commutative
	diagram of $\pi$-formal group schemes
	$$
	\xymatrix{
		0 \ar[r] & N^{[m]n}(\hG) \ar[r]^{\iota_m}\ar[d]_{\rotatebox{90}{$\sim$}} & J^{m+n}(\hG) \ar[r]^{u}
		\ar[d]_{\rotatebox{90}{$\sim$}} & J^m(\hG) \ar[r] 
		\ar[d]_{\rotatebox{90}{$\sim$}} & 0 \\
		0 \ar[r] & \bb{W}_{n-1} \ar[r]^{V^{\circ {(m+1)}}} & \bb{W}_{m+n} \ar[r]^T & 
		\bb{W}_m \ar[r] & 0. 
	}$$
	
	In this particular case of $G= \hG$, the Witt vector Frobenius map $F$ has 
	two different incarnations:
	
	$(1)$ The lift of Frobenius map $\phi: J^{m+n}(\hG) \simeq \bb{W}_{m+n} 
	\map \bb{W}_{m-1+n} \simeq J^{m-1+n}(\hG)$ is given by the Witt vector 
	Frobenius $F:\bb{W}_{m+n} \map \bb{W}_{m-1+n}$.
	
	$(2)$ The lateral Frobenius map $\fra_m: N^{[m]n}(\hG) \simeq \bb{W}_{n-1}
	\map \bb{W}_{n-2} \simeq N^{[m]n-1}(\hG)$ (as in \cite{BoSa1, PaSa}) also identifies with the Witt vector
	Frobenius $F: \bb{W}_{n-1} \map \bb{W}_{n-2}$.
	
	The following table compares the maps and identities 
	satisfied in the special case of $\hG$ and the general case of
	a $\pi$-formal group scheme.
	
	\vspace{.25cm}
	
	\begin{tabular}{|l|l|l|}
		\hline
		& Case of $G=\hG$ & General case of group scheme $G$ \\
		\hline
		Morphisms & $(i)$ Verschiebung map $V^{\circ (m+1)}: \bb{W}_{n-1} \map \bb{W}_{m+n} $ 
		& $(i)$ $\iota_m: N^{[m]n}G \map J^{m+n}G$ \\
		& $(ii)$ Frobenius map $F:\bb{W}_{m+n} \map \bb{W}_{m+n-1}$ 
		& $(ii)$ $\phi:J^{m+n}G \map J^{m+n-1}G$ and \\
		& \hspace{.5cm} Frobenius map $F: \bb{W}_{n-1} \map \bb{W}_{n-2}$ & \hspace{.5cm} $\fra_m:N^{[m]n}G 
		\map N^{[m]n-1}G$\\
		& $(iii)$ Multiplication by $\pi: \bb{W}_{n-1} \map \bb{W}_{n-1}$  
		& $(iii)$ $\phm: N^{[m]n}G \map N^{[m-1]n}G$ \\
		\hline
		Identities & $(i)$ For $m \geq 0$, & $(i)$ For $m \geq 0$, \\
		& $~~~$ $F \circ V^{\circ (m+1)} = V^{\circ m} \circ \pi $ & $~~~$
		$\phi \circ \iota_m = \iota_{m-1} \circ \phm$ \\
		& $(ii)$ For $n \geq 2$, & $(ii)$ For $n \geq 2$, \\
		& $~~~$ $F^{\circ (m+n)} \circ V^{\circ (m+1)} = F^{\circ (m-1+n)} \circ 
		V^{\circ (m+1)} \circ F$ 
		& $~~~$ $\phi^{\circ (m+n)} \circ \iota_m = 
		\phi^{\circ (m-1+n)}\circ \iota_m \circ \fra_m$ \\
		& $(iii)$ $\pi \circ F = F \circ \pi$ & $(iii)$ $ \phm \circ \fra_m = 
		\fra_m \circ \phm$ \\
		\hline
	\end{tabular}
	
	\vspace{.3cm}

\textbf{Acknowledgments.} The first author would like to thank CSIR for financial support under the scheme 09/1031(0011)/2021-EMR-I. The second author would like to  thank the Lodha Mathematical Sciences Institute (LMSI) for their hospitality and support.  The first author was also partially supported by the postdoctoral fellowship at the Tata Institute of Fundamental Research (TIFR). 
\section{Prelimineries on Arithmetic Jet Spaces}\label{AJS}
Here we recall some of the basic theory of arithmetic jet spaces.
Let $\mO$ be a Dedekind domain of characteristic $0$ with finite residue fields. Fix a non-zero prime ideal $\mathfrak{p}$, let $k$ as its residue field and let $q$ be the cardinality of $k$, that is $q=p^h$ for some prime $p$ and $h\geq 1$. Let $\pi$ be one of the uniformizers of $\mO_{\mathfrak{p}}$. Let $B$ be an $\mO$-algebra and let $A$ be a $B$-algebra, with structure map $u:B\rightarrow A.$ We define a $\pi$-derivation $\delta$ as a set-theoretic map $\delta:B\rightarrow A$ that satisfies for all $x,y\in B,$
\begin{enumerate}
	\item $\delta(1)=0$
	\item \red{$\delta(x+y)=\delta(x)+\delta(y)+C_\pi(u(x),u(y))$}
	\item $\delta(xy)=u(x)^q\delta y+u(y)^q\delta x+\pi\delta x\delta y $
\end{enumerate}
where 
$$ C_\pi(s,t)= 
\frac{s^q+t^q-(s+t)^q}{\pi}
$$
for all $s,t \in A$.
Given such a $\pi$-derivation $\delta$, we define the lift of Frobenius \red{$\phi(x):=u(x)^q+\pi\delta x$} which is a ring homomorphism satisfying \red{$$\phi(x)\equiv u(x)^q \mod \fp.$$ }
Then $\Ou$ has a unique $\pi$-derivation $\delta$ with respect to the identity map given by
$$
\delta (x) = \frac{x - x^q}{\pi}
$$
for all $x \in \Ou$ such that the associated lift of Frobenius $\phi = \mathbbm{1}$ on $\Ou$.

Fix a \red{$\pi$-adically complete} $\mO$-algebra $R$ with a $\pi$-derivation $\delta$ and \red{ set $S=\Spf R$}.
Let $X$ and $Y$ be $\pi$-formal schemes over $S$. We say $(u,\delta)$ is a prolongation and we write $Y\xrightarrow{(u,\delta)}X$, if $u:Y\rightarrow X$ is a map of $S$-schemes and 
$\delta:\mO_X\rightarrow u_*\mO_Y$ is a $\pi$-derivation making the following diagram commute:
$$ 
\xymatrix{
	R\ar[r] &u_*\mO_Y\\
	R\ar[u]^{\delta}\ar[r] & \mO_X.\ar[u]_{\delta}
}$$
\red{We define a \textit{prolongation sequence} to be a sequence of $\pi$-formal schemes
	$$S\xleftarrow{(u_0,\delta_0)}T^0\xleftarrow{(u_1,\delta_1)}T^1\xleftarrow{(u_2,\delta_2)}\cdots,$$
	where for each $i$, $T^{i-1}\xleftarrow{(u_i,\delta_i)}T^{i}$ is a prolongation satisfying 
	$$u_{i+1}^*\circ\delta_i=\delta_{i+1}\circ u_i^*$$ 
	as set-theoretic maps between sheaves 
	$\Ou_{T^i} \map (u_{i+1})_* (u_{i})_* \Ou_{T^{i+2}}$
	where $u_{i+1}^*: \Ou_{T^i} \map (u_{i+1})_* \Ou_{T^{i+1}}$ is the pull back morphism of sheaves induced by $u_{i+1}$. To simplify  notations, 
	we shall omit the subscripts on $u_i$ and $\delta_i$ whenever no confusion is likely to arise.
}

We will denote a prolongation sequence as $T^*$ or $\{T^n\}_{n\geq 0}.$
Prolongation sequences naturally
form a category $\mathcal{C}$. Let $S^*$ be the {prolongation sequence} defined by $S^i = \Spf R$ for all $i, u=$id and $\delta$ is the fixed $\pi$-derivation on $R$. Then let $\mathcal{C}_{S_*}$ denote the category
of prolongation sequences defined over $S^*$. For any $\pi$-formal $S$-scheme $X$ and for all $n \geq 0$
we define the $n$-th jet space functor $J^nX$ as \red{$$J^nX(B):=X(W_n(B))=\Hom_S(\Spf (W_n(B)),X),$$}
for any $R$-algebra $B$, where $W_n(B)$ is the ring of $\pi$-typical Witt vectors of length $n+1$.

\red{Let $A$ be a $\pi$-adically complete $R$-algebra with presentation	$
	A=R[\bx]\h/I,
	$
	where $\bx=(x_\gamma)_{\gamma\in \Gamma}$ is a (possibly infinite) collection of variables indexed by $\Gamma$ and $I$ is a $\pi$-adically complete 
	ideal of $R[\bx]\h$.
	Following \cite[Section 1.2, p.~315]{Bu2}, we define for all $n\geq 0$
	$$
	J_nA= \frac{R[\bx,\bx',\ldots,\bx^{(n)}]\h}
	{(I,\delta I,\ldots,\delta^{n}I)\h},
	$$
	where for each $i$, $\bx^{(i)}$ denotes the collection of indeterminates
	$ \bx^{(i)}=(x_\gamma^{(i)})_{\gamma\in \Gamma}$, $(I,\delta I, \ldots,
\delta^{n} I)\h$ denotes the $\pi$-adic completion of the ideal $(I,\delta I,
\ldots, \delta^n I)$ inside $R[\bx,\bx', \ldots, \bx^{(n)}]\h$ such that 
	$\delta^{i} I$ is as defined in Section 1.2 of \cite{Bu2}.
	Moreover, if $X$ is an affine $\pi$-formal scheme over $S$, say
	$
	X=\Spf A
	$
	for some $\pi$-adically complete $R$-algebra $A$, then
	$
	J^nX \simeq \Spf\bigl({J_nA}\bigr). 
	$}

\red{Note that, for every $\pi$-formal scheme $X$, there is a natural projection $u:J^nX\to J^{n-1}X$ and a $\pi$-derivation $\delta:\mathcal{O}_{J^{n-1}X}\to u_*\mathcal{O}_{J^{n}X}$ for each $n\geq 1$. These endow the sequence of $\pi$-formal schemes
	$$
	J^*X:=\left\{ X\xleftarrow{(u,\delta)}J^1X\xleftarrow{(u,\delta)}J^2X\xleftarrow{(u,\delta)}\cdots\xleftarrow{(u,\delta)}J^nX\xleftarrow{(u,\delta)}\cdots
	\right\}
	$$
	with the structure of a prolongation sequence, which is called the 
	{\it canonical prolongation sequence} of $X$. The functor $J^*$ is the right
	adjoint to the natural forgetful functor from $\mcal{C}$ to $\pi$-formal schemes 
	given by $T^* \mapsto T^0$ for all $T^* \in \mcal{C}$
	(see Proposition $1.1$ in \cite{Bu3})
	\begin{align*}
		\Hom_{\mathcal{C}}(T^*,J^*X) \simeq \Hom_{\mcal{D}}(T^0,X),
	\end{align*}
	where $\mcal{D}$ is the category of $\pi$-formal schemes over $S^0$.
	To simplify notations, we shall write $u:J^nX\to J^mX$ for the natural projection whenever $0\leq m\leq n$. }

We briefly recall the relevant definitions and the associated morphisms that will be used throughout the paper. For further details, we refer the reader to Section 3 of \cite{BoSa1} and Section 2 of \cite{BeMaSa}.
Let $B$ be a $\pi$-adically complete $R$-algebra, and for all $n\geq 0$, we write $B^{\phi^n}$ for the ring $B$ with the structure map given by the composition $R\xrightarrow{\phi^n}R\rightarrow B$ and define 
$\prod_{\phi}^nB:=B\times B^\phi\times\cdots \times B^{\phi^n}$. 
As a set we define $$W_n(B):=B^{n+1},$$ to be the set of Witt vectors of length $n+1$ and define the set-theoretic map $w:W_n(B)\rightarrow \prod_\phi^nB$ by $w(x_0,\cdots,x_n)=\langle w_0,\cdots,w_n\rangle$ where for all $i\geq 0$, $$w_i=x_0^{q^i}+\pi x_1^{q^{i-1}}+\cdots+\pi^i x_i$$ are the Witt polynomials and $w$ is known as the ghost map. We can define the ring $W_n(B)$, the ring of truncated $\pi$-typical Witt vectors by Theorem 3.1 of \cite{BoSa1}, which says that, for every $n\geq 0$, there exists a unique functorial $R$-algebra structure on $W_n(B)$ such that $w$ becomes a natural transformation of functors of $R$-algebras.
Furthermore, let us recall some of the important maps on 
the ghost side: 
\begin{enumerate}
	\item The truncation ring map $T_w:\prod_\phi^nB\rightarrow \prod_\phi^{n-1}B$ given by $T_w\langle b_0,\cdots, b_n\rangle=\langle b_0,\cdots b_{n-1}\rangle$.
	\item The right shift operator multiplied by $\pi$,  $V_w:\prod_\phi^nB^\phi\rightarrow\prod_\phi^{n+1}B$ given by $V_w\langle b_0,\cdots,b_n\rangle=\langle 0,\pi b_0,\cdots,\pi b_n\rangle$. Note that $V_w$ is an additive map only.
	\item  The left shift Frobenius ring map $F_w:\prod_\phi^nB\rightarrow \prod_\phi^{n-1}B^\phi$ given by $F_w\langle b_0,\cdots, b_n\rangle=\langle b_1,\cdots b_{n}\rangle$.
	
\end{enumerate}

The above maps on the ghost side functorially induce the following on the 
$\pi$-typical Witt vectors:
\begin{enumerate}
	\item The {\it truncation} or the {\it restriction}
 map $T: W_n(B) \map W_{n-1}(B)$ given by $T(b_0, \dots, b_n)
	= (b_0,\dots ,b_{n-1})$ is an $R$-algebra map satisfying 
	$ w \circ T = T_w \circ w.  $
	\item  
	The additive {\it Verschiebung} map $V:W_n(B^\phi)\rightarrow W_{n+1}(B)$ given by 
$$V(b_0,\cdots, b_n)=(0,b_0,\cdots, b_n)$$ 
satisfies $w \circ V = V_w \circ w$. 
	\item
	The {\it Frobenius} ring homomorphism $F:W_n(B)\rightarrow 
W_{n-1}(B^\phi)$ satisfying $w \circ F = F_w \circ w$
	such that if 
	\begin{equation}\label{Frob_Witt}
		F(b_0,\cdots,b_n)=(F_0(b_0,b_1),\cdots,F_{n-1}(b_0,\cdots,b_n)),
	\end{equation}
	for all $0\leq i\leq n-1$, then $F_i$ are integral polynomials in $(i+2)$ 
	variables, in particular $F_i\in \ZZ[\pi][\bx_0,\cdots,\bx_i]$. 
\end{enumerate}

Note that $FV$ is the multiplication by $\pi$ map on $W_n(B^\phi).$

\color{black}
Given an $S$-scheme $X$, we define $X^{\phi^n}(B):=X(B^{\phi^n})$ for any $R$-algebra $B$. Hence, the above functor can be represented by the base change of $X$ over the map $\phi^n:S\rightarrow S$ given by $X^{\phi^n}=X\times_{S,\phi^n}S.$ Next we define 
${\prod}_\phi^nX:=X\times_S X^\phi\times_S \cdots \times_S X^{\phi^n}.$
Then for any $R$-algebra $B$, we have $\prod_\phi^nX(B)=X(B)\times X(B^\phi)\times\cdots\times X(B^{\phi^n})$. \red{The ghost map $w:W_n(B)\to \prod_\phi^n B$ is a map of $R$-algebras which induces a morphism of $\pi$-formal $S$-schemes (still denoted as $w$)  $$w:J^nX\rightarrow {\prod}_\phi^n X,$$ 
	since $X$ is an $S$-scheme.}
Note that $w$ is injective when evaluated on points with coordinates in any flat $R$-algebra. The operators $F$ and $F_w$ induce the map $\phi$ and $\phi_w$ fitting into the following commutative diagram
\begin{equation*}
	\xymatrix{
		J^nX\ar[r]^w\ar[d]_\phi & \prod_\phi^n X\ar[d]_{\phi_w}\\
		J^{n-1}X\ar[r]^w & \prod_\phi^{n-1}X.
	}
\end{equation*}
Here the map $\phi_w$ can be defined as the following left-shift operator 
$$\phi_w\langle w_0,\cdots,w_n\rangle=\langle \phi_S(w_1),\cdots, \phi_s(w_n)\rangle,$$ 
where $\phi_S:X^{\phi^i}\rightarrow X^{\phi^{i-1}}$ is the composite of the morphisms
in the following diagram
\begin{equation*}
	\xymatrix{
		X^{\phi^i}\ar[r]^{\hspace{-0.25in}\sim} & X^{\phi^i}\times_{S,\phi}S\ar[r]\ar[d] & X^{\phi^{i-1}}\ar[d]\\
		& S\ar[r]^{\red{\phi_S}} & S.
	}
\end{equation*}

\section{Shifted Witt Vectors}\label{SWV}
\red{
	In this section, we first recall the notion of the general $m$-shifted Witt vectors introduced in \cite{BoSa1} and \cite{PaSa} and some of their properties. We then define our map $E_{[m]}$ on such shifted Witt vectors.}

\red{Let $B$ be a $\pi$-adically complete $R$-algebra with structure map $f:
	R\to B$. 
	Fix an $m\in\mathbb{N}\cup \{0\}$. 
	This induces a natural ring map $W_m(f): W_m(R) \map W_m(B)$.
	For each $n\geq 0$, let $T:W_{m+n}(B)\to W_m(B)$ be the canonical truncation map 
	and consider the fibre product of these ring maps:
	$$
	\xymatrix{
		W_{[m]n}(B):=W_m(R)\times_{W_m(B)}W_{m+n}(B)\ar[rr]^-{I_{[m]n}}\ar[d] && W_{m+n}(B)\ar[d]^T\\
		W_m(R)\ar[rr]^-{W_m(f)} && W_m(B).
	}
	$$
	We refer to the $R$-algebra $W_{[m]n}(B)$ as the ring of \textit{$m$-shifted Witt vectors} of length $m+n+1$.
	Note that there is a natural set-theoretic bijection  $$R^{m+1}\times B^n\rightarrow W_{[m]n}(B)$$ given by $$((r_0,\cdots,r_m),(b_{m+1},\cdots, b_{m+n}))\mapsto ((r_0,\cdots,r_m),(f(r_0),\cdots,f(r_m),b_{m+1},\cdots,b_{m+n})),$$ for 
	all $r_i\in R$, $b_j\in B$ and $i,j\geq 0$.
	Define the product ring $$\prod_{[m]n}B:=\prod_{\phi}^{m}R\times_{\prod_{\phi}^{m}B}\prod_{\phi}^{m+n}B=(R\times R^\phi\times\cdots\times R^{\phi^m})\times (B^{\phi^{m+1}}\times\cdots\times B^{\phi^{m+n}}).$$  
	Hence, by the universal property of fibre products, we have an induced ghost 
	map for the $m$-shifted Witt vectors
	$$
	w:W_{[m]n}(B)\longrightarrow\prod_{[m]n}B.
	$$ We define the following ring homomorphisms on the ghost components.
	\begin{enumerate}
		\item[(1)] The truncation map $T_w:\prod_{[m]n}B\rightarrow \prod_{[m]n-1}B$ as 
		\begin{equation*}\label{T_w}
			T_w\langle z_0,\cdots , z_{m+n}\rangle=\langle z_0,\ldots,z_{m+n-1}\rangle.
		\end{equation*}
		\item[(2)] The $m$-th Frobenius map ${F}_{[m]_w}:\prod_{[m]n}B\rightarrow \prod_{[m]n-1}B^\phi$ as 
		\begin{equation*}\label{F_w}
			{F}_{[m]_w}\langle z_0,\cdots ,z_m,z_{m+1},\cdots, z_{m+n}		\rangle=\langle\phi(z_0),\ldots,\phi(z_m),z_{m+2},\ldots,z_{m+n}\rangle.
		\end{equation*}
		Observe that the $(m+1)$-st component is omitted.
		\item[(3)] The left shift map $E_{[m]_w}:\prod_{[m]n}B\rightarrow \prod_{[m-1]n}B^\phi$ as
		\begin{equation}\label{F_m_w}
			E_{[m]_w}\lan z_0,\cdots ,z_m,z_{m+1},\cdots, z_{m+n}\ran =\lan z_1,\cdots, z_{m},z_{m+1},\cdots, z_{m+n}\ran .
		\end{equation} 
	\end{enumerate}
	Note that the map $T : W_{[m]n}(B) \map W_{[m](n-1)}(B)$ given by 	
$$T(x_0,\cdots,x_m,x_{m+1},\cdots,x_{m+n})=(x_0,\cdots,x_m,x_{m+1},\cdots,x_{m+n-1})$$
 is an $R$-algebra map satisfying 
	$$
	w \circ T = T \circ w.
	$$
	There exists a ring map $F_{[m]}:W_{[m]n}(B)\rightarrow W_{[m](n-1)}(B)$, called 
	the {\it lateral Frobenius} on the $m$-shifted Witt vectors that is compatible with 
	 ${F_{[m]}}_w$ as shown in Theorem $5.1$ of \cite{PaSa}.
	\begin{theorem}
		\label{lat_Frob1}
		There exists a ring homomorphism ${F}_{[m]}:W_{[m]n}(B)\rightarrow W_{[m]n-1}(B^\phi)$ such that 
		$$
		\xymatrix{
			W_{[m]n}(B)\ar[r]^w\ar[d]_{{F}_{[m]}} & \prod_{[m]n}B\ar[d]^{{{F}}_{[m]_w}}\\
			W_{[m]n-1}(B^\phi)\ar[r]^w & \prod_{[m]n-1}B^\phi
		}
		$$
		commutes. Moreover if ${F}_{[m]}(x_0,\cdots,x_{m+n})=(\tilde{F}_0,\cdots,\tilde{F}_{m+n-1})$, then for all $0\leq h\leq m+n-1$ we have $$\tilde{F}_h\equiv x_h^q \mod \pi.$$
	\end{theorem}
	\begin{proof}
		\red{	We refer to Theorem 5.1 in \cite{PaSa} for the proof.}
	\end{proof}
}

Consider the $\pi$-typical Witt vectors $W_n(B)$ of length $n+1$ where 
$B$ is a $\pi$-adically complete  $R$-algebra and the lift of Frobenius $F: W_n(B) \map 
W_{n-1}(B^\phi)$ is \red{defined as in (\ref{Frob_Witt})}. 
We define the map $\Em:\Wm{n}(B) \map W_{[m-1]n}(B^\phi)$ as 
\begin{align*}
	\Em(x)  = (F_0(x),\dots ,F_{m-1+n}(x)),
\end{align*}
for any $x = ((r_0,\dots, r_m),( a_{m+1}, \dots, a_{m-1+n})) \in \Wm{n}(B)$.  \red{Observe that the morphism
	$$
	E_{[m]w}:\prod_{[m]n}B\longrightarrow\prod_{[m-1]n}B^\phi
	$$
	is induced by the left shift Frobenius operators
	$$
	F_w:\prod_\phi^mR\longrightarrow\prod_\phi^{m-1}R^\phi
	\quad\text{and}\quad
	F_w:\prod_\phi^{m+n}B\longrightarrow\prod_\phi^{m+n-1}B^\phi.
	$$
	Similarly, the morphism
	$$
	E_{[m]}:W_{[m]n}(B)=W_m(R)\times_{W_m(B)}W_{m+n}(B)
	\longrightarrow
	W_{[m-1]n}(B^\phi)=W_{m-1}(R^\phi)\times_{W_{m-1}(B^\phi)}W_{m+n-1}(B^\phi)
	$$
	is induced by the Frobenius morphisms
	$$
	F:W_m(R)\longrightarrow W_{m-1}(R^\phi)
	\quad\text{and}\quad
	F:W_{m+n}(B)\longrightarrow W_{m+n-1}(B^\phi).
	$$
	Hence, $E_{[m]}$ is well defined by the universal property of the fibre product. Consequently, we obtain the following commutative diagram of rings:
	\begin{align}\label{F_m} \xymatrix{ \Wm{n}(B)\ar[d]_{\Em} \ar[r]^{\Ima{n}} & W_{m+n}(B) \ar[d]^F\\ W_{[m-1]n}(B^\phi) \ar[r]_{I_{[m-1]n}} & W_{m-1+n}(B^\phi). } \end{align}
}

\begin{proposition}
	The ring homomorphism $E_{[m]}:W_{[m]n}(B)\rightarrow W_{[m-1]n}(B^\phi)$ commutes with the ghost map and $E_{[m]_w}$ as defined in (\ref{F_m_w}). More precisely, the following diagram of rings commutes:
	\begin{equation*}
		\xymatrix{
			W_{[m]n}(B)\ar[r]^w\ar[d]_{E_{[m]}} & \prod_{[m]n}B\ar[d]^{E_{[m]_w}}\\
			W_{[m-1]n}(B^\phi)\ar[r]^w & \prod_{[m-1]n}B^\phi.
		}
	\end{equation*}
\end{proposition}
\red{
	\begin{proof}
		Since $W_{[m]n}(B)$ and $W_{[m-1]n}(B^\phi)$ are defined as fibre products, the commutative diagram above may equivalently be written as
		$$
		\xymatrix{
			W_m(R)\times_{W_m(B)}W_{m+n}(B)\ar[d]^-\Em\ar[r]^-w & \prod_{\phi}^mR\times_{\prod_{\phi}^mB}\prod_{\phi}^{m+n}B\ar[d]_{{\Em}_w}\\ W_{m-1}(R)\times_{W_{m-1}(B^\phi)}W_{m-1+n}(B^\phi)\ar[r]^-w & \prod_{\phi}^{m-1}R\times_{\prod_{\phi}^{m-1}B^\phi}\prod_{\phi}^{m-1+n}B^\phi.
		}
		$$
		Here, the morphisms $\Em$ and ${\Em}_w$ are induced by the Frobenius morphisms on Witt vectors and the left shift Frobenius operators on ghost components, respectively. Hence it follows immediately from the universal property of the fibre product that the above diagram is commutative.
	\end{proof}
}\color{black}

\section{Kernels of Jet Spaces}\label{KJS}

In this section we will prove our main theorems. 
We first make some preparations.
Let $(X,P)$ denote a $\pi$-formal scheme $X$ with a marked $R$-point $P\in X(R)$. Since $R$ 
is equipped with a $\pi$-derivation $\delta$ on it, the universal property
of $\pi$-typical Witt vectors induces the canonical ring map $\exp_\delta:
R \map W_n(R)$ and therefore the corresponding map $\Spf W_n(R) \map \Spf R$
of $\pi$-formal schemes. Hence by composing with $P$, 
we naturally obtain a marked $R$-point $P^n$ on $J^nX$, which we will denote by $(J^nX,P^n)$. If $X=\Spf A$, let $\iota^*:A\rightarrow R$ denote the ring map associated to $P$. If we denote the ring map of $P^n$ as $\iota_n^*:A\rightarrow W_n(R)$, then we have $\iota_n^*=\exp_\delta\circ \iota^*.$ 

For each $n$, we have the following commutative diagram
$$
\xymatrix{
	J^{n+1}X \ar[ddr] \ar@/^1pc/[rrd]^\phi \ar[rd]^{\phi_{\mathrm{rel}}}
	&  & \\
	& J^nX \times_{S,\phi} S \ar[d] \ar[r] & J^nX \ar[d] \\
	& S \ar@/_1pc/[u]_{P^n_{\phi}} \ar@/^2pc/[uul]^{P^{n+1}} \ar[r]^\phi & 
	S \ar@/_1pc/[u]_{P^n}
}
$$
where $P^n_{\phi} \in J^nX \times_{S,\phi} S (S) \simeq J^nX(S^\phi)$ is the 
pullback of the section $P^n$ via $\phi$.

For all $n\geq 0$, we define $N^{[m]n}X:=J^{m+n}X\times_{J^mX}S$ which is the following fibre product
$$
\xymatrix{
	J^{m+n}X\ar[d]_{u} & N^{[m]n}X\ar[d]\ar[l]_{\iota_m}\\
	J^mX & S.\ar[l]^{P^m}
}
$$
Then clearly $N^{[m]n}X=\Spf N_{[m]n}A$, where $N_{[m]n}A=J_{m+n}A\otimes_{J_mA}R$. Also functorially $N^{[m]n}X $ can be described as 
\begin{align}
	\label{g_map}
	N^{[m]n}X(B)=\{g:A\rightarrow W_{[m]n}(B)|&\text{ where if }g=(g_0,\cdots,g_m,\cdots, g_{m+n}),\notag\\& \text{ then }(g_0,\cdots,g_m)=i_{m}^*\}.
\end{align}
Note that the usual projection map \red{$u:J^{m+n}X\rightarrow J^{m-1+n}X$ naturally induces $u:N^{[m]n}X\rightarrow N^{[m]n-1}X$ for all $n\geq 1$ and $m \geq 0.$} 

We define the generalized lateral Frobenius,  $\mf{f}_m:N^{[m]n}X\rightarrow N^{[m]n-1}X$ as $$\mf{f}_m(r)=r\circ{F}_{[m]} \text{ for all } r\in N^{[m]n}X(B)
\subset \Hom (\Spf W_{[m]n}(B), X),$$ where $B$ is any $\pi$-adically complete $R$-algebra and denote (by a 
slight abuse
of notation) ${F}_{[m]}:\Spf W_{[m]n-1}(B^\phi)\rightarrow \Spf W_{[m]n}(B)$ as the morphism 
induced by the lateral Frobenius ring map $F_{[m]}:W_{[m]n}(B)
\rightarrow W_{[m]n-1}(B^\phi)$.


By Theorem 5.3 in \cite{PaSa}, it is shown that $\mf{f}_m$ is a lift of Frobenius and it satisfies 
\begin{equation*}\label{lat_Frob}
	\phi^{m+j}\circ \iota_m=\phi^{m+j-1}\circ \iota_m\circ \mf{f}_m,
\end{equation*}
for all $2\leq j\leq n.$\\

\cblue{

	For any $\pi$-adic $R$-algebra $B$, the
	ring map $E_{[m]}: W_{[m]n}(B) \map W_{[m-1]n}(B^\phi)$ induces the map
	\begin{align*}
		\phm : X(W_{[m]n}(B)) \map X(W_{[m-1]n}(B^\phi))
	\end{align*}
	which satisfies the following
	$$
	\xymatrix{
		X(\Wm{n}(B)) \ar[rr]^{X(E_{[m]})=:\phm}  \ar[d]_{u} & & X(W_{[m-1]n}(B^\phi))
		\ar[d]^{u}\\
		X(W_m(R)) \ar[rr]^{X(F) =: \phi}  \ar[d] & & X(W_{m-1}(R^\phi)) \ar[d] \\
		S(R) \ar@/^1pc/[u]^{P^m} \ar[rr]^\phi & & S(R^\phi). \ar@/_1pc/[u]_{P^{m-1}}
	}
	$$
	Hence by (\ref{g_map}),
	for any $\pi$-formal $R$-algebra $B$, we have the morphism 
	\begin{align*}
		\label{Phi_[m]}
		\phm : N^{[m]n}X(B) \map X(W_{[m-1]n}(B^\phi))
	\end{align*}
	that induces a morphism $\phm : N^{[m]n}X \map N^{[m-1]n}X$ over $\phi : S 
	\map S$.
}

\begin{theorem}
	\label{JnNm}
	Let $X$ be a $\pi$-formal scheme over $\Spf R,$ then for all $n\geq 1$ we have $$N^{[m]n}X\simeq J^{n-1}(N^{[m]1}X).$$
\end{theorem}
\begin{proof}
	\red{	We refer to \cite{PaSa} Theorem 1.3 for the proof.}
\end{proof}




Let $(X,P)$ be an $R$-pointed smooth $\pi$-formal scheme of relative dimension $g$ over $S$.  We define \cy{$${\prod}_{[m]n}X:=(S\times_{X,\phi} S^\phi\times_{X,\phi} \cdots \times_{X,\phi} S^{\phi^m})\times(X^{\phi^{m+1}}\times \cdots \times X^{\phi^{m+n}}) $$ where the maps from $S^{\phi^i}\xrightarrow{P} X$ is the composition of $S\xrightarrow{\phi^i} S\xrightarrow{P} X$ and for all $0\leq i\leq m$.}  Hence any point $z\in \prod_{[m]n}X$ can be expressed as $z=\langle P,\cdots,P,z_{m+1},\cdots,z_{m+n}\rangle$, where $z_i\in X^{\phi^i}$.

We have the associated morphism $\mf{f}_{m_w}:\prod_{[m]n}X\rightarrow \prod_{[m]n-1}X$ corresponding to  $\mf{f}_{m_w}$ respectively, given by 
\begin{align*}
	\mf{f}_{m_w}\langle P,\cdots,P,z_{m+1},\cdots,z_{m+n}\rangle&= \langle P,\cdots,P,\phi_S(z_{m+2}),\cdots,\phi_S(z_{m+n})\rangle.
\end{align*}
Hence we have the following morphism of $\pi$-formal  schemes
\begin{equation}\label{iota_m}
	\xymatrix{
		N^{[m]n}X\ar[r]^{\iota_m}\ar[d]_w & J^{m+n}X\ar[r]^{u}\ar[d]_w & J^mX\ar[d]_w\\
		\prod_{[m]n}X\ar[r]^{{\iota_m}_w} & \prod_\phi^{m+n}X\ar[r]^{{u}_w} & \prod_\phi^m X,
	}
\end{equation}
where $u:J^{m+n}X\rightarrow J^mX$ is the composition of the natural projection map $u:J^{m+n}X\rightarrow J^{m+n-1}X$, $n$ times to descend down to $J^mX$ and in the ghost side, 
\begin{align*}
	{\iota_m}_w\langle P,\cdots,P,z_{m+1},\cdots,z_{m+n}\rangle & =\langle P,\cdots,P,z_{m+1},\cdots,z_{m+n}\rangle,\notag\\
	 {u}_w\langle z_0,\cdots,z_{m+n}\rangle & =\langle z_{0},\cdots,z_{m}\rangle.
\end{align*}	
We have a closed immersion of $\pi$-formal schemes $i_m:\Spf W_m(R)\longrightarrow \Spf W_{[m]n}(B),$ for any $\pi$-adically complete $R$-algebra $B$. Therefore, given a $\pi$-adically complete $R$-algebra $B$, extending the expression in (\ref{g_map}) for any group scheme,  we have 
\begin{equation}
	\label{N_mnG}
	N^{[m]n}X(B)=\{f\in X(W_{[m]n}(B))|f\circ i_m=P^m\}.
\end{equation}

\begin{lemma}\label{Affine_Nmn}
	Let $(X,P)$ be as above. Suppose $U$ is an affine open subscheme of $X$ containing $P$ such that there exists an \'etale morphism $v:U\rightarrow H$, where $H=\Spf R[\bx_0]^{\hat{•}}$ and $\bx_0=(x_{01},\cdots, x_{0g})$ are coordinates such that $v(P)$ in $H$ is given by setting $x_{01}=\cdots=x_{0g}=0.$ Then $$N^{[m]n}X\simeq N^{[m]n}H.$$
\end{lemma}

\begin{proof}
	We have $P\in U$ and $U$ is an affine open subscheme of $X$. This implies $P$ is an $R$-marked point of $U$ as well and $v\circ P$ gives us a morphism $S\rightarrow H.$	Hence we obtain
	\begin{align}\label{Witt cord}
		N^{[m]n}X& =J^{m+n}X\times_{J^m X} S\simeq J^{m+n}U\times_{J^mU} S.
	\end{align}
	\red{Note that $v:U\rightarrow H$	is an \'etale morphism, hence by  \cite{Bu4}(Corollary 3.16, Page 77), we have for all $r\geq 0$,
		$$ 
		J^{r}U\simeq J^{r}H\times_{H}U.
		$$
		This implies
		\begin{align*}
			J^{m+n}U&\simeq J^{m+n} H\times_H U\\
			&\simeq J^{m+n}H\times_{J^mH}(J^{m}H\times_H U)\\
			&\simeq J^{m+n}H\times_{J^mH}J^mU.
		\end{align*}
		Finally from (\ref{Witt cord}), we obtain 
		\begin{align*}
			N^{[m]n}X&\simeq J^{m+n}U\times_{J^mU} S\\
			&\simeq (J^{m+n} H\times_{J^mH}J^mU)\times_{J^m U}S\\
			&\simeq J^{m+n}H\times_{J^mH}S\simeq N^{[m]n}H.
		\end{align*}
	}
\end{proof}

Let $\mathbb{W}_n$ denote the Witt vectors of length $n+1$ as a group object in the category of $\pi$-formal schemes. Then $\mathbb{W}_n\simeq \Spf R[x_0,\cdots,x_n]^{\hat{}}$ and the ring $R[x_0,\cdots,x_n]^{\hat{}}$ has a co-algebra structure induced from the Witt vector addition.

\color{black}
Let us assume $G$ to be a smooth commutative $\pi$-formal group scheme over $S$ of relative dimension $g$. Let $\bx_0=(x_{01},\cdots,x_{0g})$ be a system of \'etale coordinates at the 
identity section $e$ of $G$ and define $H=\Spf R[\bx_0]^{\hat{}}$. Then we have an \'etale morphism of $\pi$-formal schemes $U\rightarrow H$ for some open affine subscheme $U$ of $G$ containing $e$.  Following the arguments on page 406 of \cite{BoSa1}, we have $J^nH\simeq \mW_n^g$ and $J^nH(B)=\Hom_R(R[\bx_0]^{\hat{}},W_n(B))$ for any $\pi$-adically complete $R$-algebra $B$. Hence $$J^nH\simeq \Spf R[\bx_0,\cdots,\bx_n]^{\hat{}}$$ where  
$(\bx_0,\cdots,\bx_n)$ is the image of $\bx_0$ in $W_n(B)$ and $\bx_i=(x_{i1},\cdots, x_{ig})$. Then $(\bx_0,\cdots,\bx_n)$
forms an \'etale coordinate system of $J^nG$ around the identity section and we will call them local Witt coordinates of length $n+1$ for $G$. For a detailed discussion we refer to page 406 of \cite{BoSa1}. From Lemma \ref{Affine_Nmn}, we have $
N^{[m]n}G\simeq N^{[m]n} H=J^{m+n}H\times_{J^mH}S.$ Here the fibre product is taken along the identity section $S\xrightarrow{e} J^mH$, which at the level of maps between $R$-algebras   $R[\bx_0,\cdots,\bx_m]\rightarrow R$, is given by $\bx_i\mapsto 0$ for all $0\leq i\leq m$. Hence we have 
\begin{equation}\label{Nmn_Witt}
	N^{[m]n}G\simeq N^{[m]n} H=J^{m+n}H\times_{J^mH}S\simeq \Spf R[\bx_{m+1},\cdots,\bx_{m+n}]^{\hat{}}.
\end{equation}

\color{black}
\begin{theorem}\label{N_mn}
	Let $G$ be a smooth commutative $\pi$-formal group scheme of relative dimension $g$. If  $v_\pi(p)\leq p-2$ then, $$N^{[m]n}G\simeq \mathbb{W}_{n-1}^g$$ as $\pi$-formal group schemes. In particular, we have the following short exact sequence of $\pi$-formal group schemes $$0\rightarrow (\mathbb{W}_{n-1})^g\rightarrow J^{m+n}G\rightarrow J^mG\rightarrow 0.$$
\end{theorem}
\begin{proof}
	\red{	We refer \cite{PaSa} Theorem 1.4 for the proof.}
\end{proof}

\color{black}

\color{red}
\color{black}

\begin{theorem}\label{Phi_m_i_m}
	The morphism of $\pi$-formal schemes $\Phi_{[m]}: N^{[m]n}X\rightarrow N^{[m-1]n}X$ commutes with the Frobenius morphism $\phi:J^{m+n}X\rightarrow J^{m-1+n}X$, that is, the following diagram commutes
	$$
	\xymatrix{
		N^{[m]n}X\ar[r]^-{\iota_m}\ar[d]_-{\Phi_{[m]}} & J^{m+n}X\ar[d]^{\phi}\\
		N^{[m-1]n}X\ar[r]^{\iota_{m-1}} & J^{m-1+n}X.
	}
	$$
\end{theorem}
\begin{proof}
	For any $\pi$-adically complete $R$-algebra $B$,
	by  (\ref{F_m}) we have the following commutative diagram
	$$
	\xymatrix{
		W_{[m]n}(B)\ar[r]^-{I_{[m]n}}\ar[d]_{E_{[m]}} & W_{m+n}(B)\ar[d]^F\\
		W_{[m-1]n}(B^\phi)\ar[r]_-{I_{[m-1]n}} & W_{m-1+n}(B^\phi).
	}
	$$
	Note that $I_{[m]n}:W_{[m]n}(B)\rightarrow W_{m+n}(B)$ induces the inclusion map $\iota_m:N^{[m]n}X(B)\rightarrow J^{m+n}X(B)$ at the level of schemes. 
	Hence applying the functor $X(-)$ on the above diagram gives our desired 
	result.
\end{proof}

Note that we have the associated morphism to ${\Phi_{[m]}}_w:\prod_{[m]n}X\rightarrow\prod_{[m-1]n}X$ corresponding to $\Phi_{[m]}$, given by
\begin{equation*}
	{\phm}_w\langle\underbrace{P,\cdots,P}_{(m+1)\text{-many}},z_{m+1},\cdots,z_{m+n}\rangle=\langle\underbrace{\phi_S(P),\cdots,\phi_S(P)}_{m\text{-many}},\phi_S(z_{m+1}),\cdots,\phi_S(z_{m+n})\rangle.
\end{equation*}

\begin{corollary}
	The following diagram is commutative
	\begin{equation*}
		\xymatrix{
			N^{[m]n}X\ar[r]^w\ar[d]_{\phm} & \prod_{[m]n}X\ar[d]_{{\phm}_w}\\
			N^{[m-1]n}X\ar[r]^w & \prod_{[m-1]n}X.	
		}
	\end{equation*}
\end{corollary}
\begin{proof}
	Consider the cube of diagrams
	\begin{equation}\label{cube}
		\xymatrix{
			& \prod_{[m]n}X\ar[dd]_>>>>>>{{\phm}_w}\ar[rr]^-{{\iota_m}_w}& & \prod_\phi^{m+n}X\ar[dd]_-{\phi_w}\\
			N^{[m]n}X\ar[ru]^w\ar[rr]^>>>>>>>>>{\iota_m}\ar[dd]^-{\phm}& & J^{m+n}X\ar[ru]^-w\ar[dd]^>>>>>>\phi &\\
			& \prod_{[m-1]n}X\ar[rr]^>>>>>>>>>{{\iota_{m-1}}_w}& & \prod_\phi^{m-1+n}X\\
			N^{[m-1]n}X\ar[ru]^-w\ar[rr]^-{\iota_{m-1}}& & J^{m-1+n}X.\ar[ru]^-w&
		}
	\end{equation}
	
	It suffices to establish the commutativity of the left face of the cube in \eqref{cube}. The top and bottom faces commute by equation \eqref{iota_m}, the front face commutes by Theorem \ref{Phi_m_i_m}, and the right face commutes by equation (3.7) of \cite{BoSa1}. Consequently, the commutativity of the left face will follow once we show that the back face commutes, that is,
	\[
	\phi_w\circ{\iota_m}_w={\iota_{m-1}}_w\circ{\phm}_w.
	\]
	
	Indeed, for $\langle\underbrace{P,\dots,P}_{(m+1)\text{-many}},z_{m+1},\dots,z_{m+n}\rangle\in \prod_{[m]n}X$, a direct computation gives
	\begin{align*}
		\phi_w\circ{\iota_m}_w\langle\underbrace{P,\cdots,P}_{(m+1)\text{-many}},&z_{m+1},\cdots,z_{m+n}\rangle=\phi_w\langle\underbrace{P,\cdots,P}_{(m+1)\text{-many}},z_{m+1},\cdots,z_{m+n}\rangle\\
		&=\langle\underbrace{\phi_S(P),\cdots,\phi_S(P)}_{m\text{-many}},\phi_S(z_{m+1}),\cdots,\phi_S(z_{m+n})\rangle\\
		&={\iota_{m-1}}_w\circ {\phm}_w\langle\underbrace{P,\cdots,P}_{(m+1)\text{-many}},z_{m+1},\cdots,z_{m+n}\rangle,
	\end{align*}
	which proves the claim, and with it the commutativity of the entire cube \eqref{cube}.
\end{proof}

Consider $\hA^N = \Spf R[\bx_0]\h$ where $\bx_0$ is an $N$-tuple of coordinate
functions and the endomorphism $\pi: \hA^N \map \hA^N$ 
induced by \cy{$\bx_0 \mapsto \pi \bx_0$}. 

\begin{lemma}
	\label{pimul}
	There exists a unique morphism  $(\pi): \prod_{i=0}^n(\hA^N) 
	\map \prod_{i=0}^n(\hA^N)$ making the following diagram commutative
	$$
	\xymatrix{
		\prod_{i=0}^n(\hA^N)\ar[d]_-{(\pi)} \ar[r]^w & \prod_{i=0}^n (\hA^N) 
		\ar[d]^-{\pi \times \cdots \times \pi}\\
		\prod_{i=0}^n(\hA^N) \ar[r]^w & \prod_{i=0}^n (\hA^N).
	}
	$$
	where $w$ is the tuple of Witt polynomials.
\end{lemma}

\begin{proof}
	It is enough to consider the case when $N=1$.
	The composition $$(\pi \times \cdots \times \pi) \circ w: \prod_{i=0}^n(\hA^1)
	\longrightarrow \prod_{i=0}^n (\hA^1)$$ is given by
	\begin{align}
		\nonumber
		(x_0,x_1, \dots, x_n) \mapsto \langle \pi x_0, \pi(x_0^q + \pi x_1),
		\cdots , \pi(x_0^{q^n} + \pi x_1^{q^{n-1}} + \cdots + \pi^n x_n) \rangle.
	\end{align}
	Consider the following system of equations in variables $L_0,\dots , L_n$ 
	given by 
	\beqar
	L_0 &=& \pi x_0 \\
	\vdots & & \vdots \\
	L_0^{q^n} + \pi L_1^{q^{n-1}} + \cdots + \pi^n L_n & = & \pi(x_0^{q^n} + \pi
	x_1^{q^{n-1}} + \cdots + \pi^n x_n).
	\eeqar
	Then clearly the above system has a unique solution and hence we obtain our 
	result by setting $(\pi) (x_0,\dots , x_n) := (L_0,\dots , L_n)$ and we are
	done.
\end{proof}

Let $H = \hA^g = \Spf R[\bx_0]^{\h}$ where $\bx_0$ is a $g$-tuple of variables. \cy{Then for any $n,m\geq 0$, we have $N^{[m]n}H\simeq\Spf R[\bx_{m+1},\cdots,\bx_{m+n}]^{\h},$ as explained in (\ref{Nmn_Witt}). }
The following composition of morphisms
\begin{align*}
	N^{[m]n}H \stk{\iota_m}{\longrightarrow} J^{m+n}H \stk{w}{\longrightarrow} 
	\prod_{\phi}^{m+n}H
\end{align*}
is given by 
\begin{align*}
	(w\circ \iota_m)(t) = \langle 0, \dots , 0, \pi^{m+1}t_0 , \dots ,
	\pi^{m+1}t_0^{q^{n-1}} + \cdots + \pi^{m+n}t_{n-1}\rangle,
\end{align*}
where $t=(t_0,\cdots,t_{n-1})\in N^{[m]n}H$ and the first $(m+1)$-entries in the \cy{right}-hand side of the above expression
are zeroes. 

Let $\pr_j: \prod_{\phi}^N H \longrightarrow \prod_{\phi}^{N-j-1} H$ be 
the projection map given by $\pr_j(y_0, \dots , y_N)= \pi^{-(j+1)}(y_{j+1},
\dots, y_N)$ whenever the above makes sense.
Consider the following composition
\begin{align*}
	\tilde{w}:
	N^{[m]n}H \stk{\iota_m}{\longrightarrow} J^{m+n}H \stk{w}{\longrightarrow} 
	\prod_{\phi}^{m+n}H \stk{\pr_m}{\longrightarrow} \prod_{\phi}^{n-1} H
\end{align*}
which is given by 
$$
(t_0,\dots, t_{n-1}) \mapsto \langle t_0, t_0^q+\pi t_1, \dots, t_0^{q^{n-1}} + 
\pi t_1^{q^{n-2}}+ \cdots + \pi^{n-1}t_{n-1}\rangle.
$$

Consider the further composition of maps
\begin{align*}
	N^{[m]n}H \stk{\iota_m}{\longrightarrow} J^{m+n}H \stk{w}{\longrightarrow} 
	\prod_{\phi}^{m+n}H \stk{\phi_w}{\longrightarrow} \prod_{\phi}^{m-1+n} H,
\end{align*}
which is given by 
\begin{align*}
	(\phi_w \circ w\circ \iota_m)(t) = \langle 0, \dots , 0, \pi^{m+1}t_0 , \dots ,
	\pi^{m+1}t_0^{q^{n-1}} + \cdots + \pi^{m+n}t_{n-1}\rangle,
\end{align*}
where the first $m$ entries in the \cy{right} hand side of the above expression
are zeroes for all $t=(t_0,\dots , t_{n-1}) \in N^{[m]n}H$. Hence by 
Theorem \ref{JnNm} and Lemma \ref{pimul} we have the following diagram
by appropriately restricting the product space of $H$ as follows
\begin{align}
	\label{dia1}
	\xymatrix{
		J^{n-1}(N^{[m]1}H) \ar[d]^-{\phm} \ar[r]^-\sim  & N^{[m]n}H \ar[d]^{\phm} 
		\ar[r]^w & \prod_{[m]n} H \ar[d]^{\phi_w} 
		\ar[r]^{\pr_m} & \prod_{\phi}^{n-1} H \ar[d]^{(\pi \times \cdots \times  \pi)}\\
		J^{n-1}(N^{[m-1]1}H) \ar[r]^-\sim
		& N^{[m-1]n}H \ar[r]^w & \prod_{[m-1]n} H \ar[r]^{\pr_{m-1}} & 
		\prod_{\phi}^{n-1} H. \\
}\end{align}
For $n =1$, the above diagram (\ref{dia1}) becomes 
\begin{align}
	\nonumber
	\xymatrix{
		N^{[m]1}H \ar[r]^-{\tilde{w}}\ar[d]_\phm & H \ar[d]^\pi\\
		N^{[m-1]1}H \ar[r]^-{\tilde{w}} & H,
}\end{align}
and hence we obtain 
\begin{align}
	\label{t0}
	\phm(t_0) = \pi t_0 \mb{ for all } t_0 \in H.
\end{align}

\begin{theorem}\label{phm&fm_comm}
	For any $\pi$-formal scheme $X$ over $S$, the lateral Frobenius $\mff_m$ and $\phm$ commute with each other. In particular, the following diagram commutes for every $m,n\geq 1$
	\begin{equation*}\label{phi_NmnX}
		\xymatrix{
			N^{[m]n}X\ar[rr]^-{\mff_m}\ar[d]_-{\phm} & & N^{[m]n-1}X\ar[d]^-{\phm}\\
			N^{[m-1]n}X\ar[rr]^-{\mff_{m-1}} & & N^{[m-1]n-1}X.
		}
	\end{equation*} 
\end{theorem}

\begin{proof}
	
	Let $B$ be a $\pi$-adically complete $R$-algebra. By (\ref{N_mnG}) it is enough to show the commutativity of the following diagram 	
	\begin{equation}\label{phi_WmnX}
		\xymatrix{
			W_{[m]n}(B)\ar[rr]^-{F_{[m]}}\ar[d]_-{E_{[m]}} & & W_{[m]n-1}(B^\phi)\ar[d]^-{E_{[m]}}\\
			W_{[m-1]n}(B^\phi)\ar[rr]^-{F_{[m-1]}} & & W_{[m-1]n-1}(B^{\phi^2}).
		}
	\end{equation} 
	
	Let us consider the following diagram with the ghost maps (\ref{phi_WmnX})
	\begin{equation}\label{phi_NmnX_w}
		\xymatrix{
			\prod_{[m]n}B\ar[rrrr]^-{{{F}}_{[m]_w}}\ar[ddd]_-{{E_{[m]}}_w} & & & & \prod_{[m]n-1}B^\phi\ar[ddd]^-{{E_{[m]}}_w} \\
			& W_{[m]n}(B)\ar[rr]^-{{{F}}_{[m]}}\ar[ul]_-w\ar[d]_-{E_{[m]}} & & W_{[m]n-1}(B^\phi)\ar[d]^-{E_{[m]}}\ar[ru]^-w &\\
			& W_{[m-1]n}(B^\phi)\ar[rr]^-{{{F}}_{[m-1]}}\ar[dl]^-w & & W_{[m-1]n-1}(B^{\phi^2})\ar[rd]_-w &\\
			\prod_{[m-1]n}B^\phi\ar[rrrr]_-{{{F}}_{{[m-1]}_w}} & & & & \prod_{[m-1]n-1}B^{\phi^2}.
		}
	\end{equation}
	By the ghost principle, it suffices to assume that $B$ has no $\pi$-torsion, in which case the ghost maps $w$ are injective. Thus  it is suffices to show the commutativity of the outer square in 
	(\ref{phi_NmnX_w}). 
	Let $$z=\langle r_0,\cdots,r_m,z_{m+1},\cdots,z_{m+n}\rangle\in \prod_{[m]n}B,$$ then 
	\begin{align*}
		{E_{[m]}}_w\circ {{F}}_{[m]_w}(z)&= {E_{[m]}}_w\circ {{F}}_{[m]_w}\langle r_0,\cdots,r_m,z_{m+1},\cdots,z_{m+n}\rangle\\
		&={E_{[m]}}_w\langle\phi(r_0),\cdots,\phi(r_m),z_{m+2},\cdots,z_{m+n}\rangle\\
		&=\langle\phi(r_1),\cdots,\phi(r_m),z_{m+2},\cdots,z_{m+n}\rangle
	\end{align*} 
	and similarly
	\begin{align*}
		{F}_{[m-1]_w}\circ {E_{[m]}}_w(z) &={F}_{[m-1]_w}\circ {E_{[m]}}_w\langle r_0,\cdots,r_m,z_{m+1},\cdots,z_{m+n}\rangle\\
		&={F}_{[m-1]_w}\langle r_1,\cdots,r_m,z_{m+1},\cdots,z_{m+n}\rangle\\
		&=\langle\phi(r_1),\cdots,\phi(r_m),z_{m+2},\cdots,z_{m+n}\rangle\\&={E_{[m]}}_w\circ {{F}}_{[m]_w}(z).
	\end{align*}
	Hence we have ${E_{[m]}}_w\circ {{F}}_{[m]_w}={F}_{[m-1]_w}\circ {E_{[m]}}_w$. Therefore, the diagram (\ref{phi_WmnX}) commutes and hence we are done.
\end{proof}

Let $G$ be a $\pi$-formal group scheme of relative dimension $g$ over $S$. 
Also assume $v_\pi \leq p-2$.
By Proposition $2.2$ and Lemma $2.3$ in \cite{Bu2}, we have an isomorphism 
 of $\pi$-formal group schemes $\Psi^{[m]}_1: N^{[m]1}G \simeq \hG^g$ 
given by $\Psi_1^{[m]}(t_0) = 1/\pi^{m+1} \log_{G,\phi^{\circ (m+1)}}
(\pi^{m+1} t_0)$ where $\log_G$ is the logarithm map from the formal group
law of $G$ to the additive group and $\log_{G,\phi^{\circ (m+1)}}$ is the
map obtained by applying $\phi^{\circ (m+1)}$ to the coefficients of $\log_G$.

\begin{theorem}
	\label{Psimiso}
Assume $v_\pi(p) \leq p-2$. Let $G$ be a $\pi$-formal smooth commutative 
group scheme of relative dimension $g$ over $S$. For all $m, n\geq 1$  we have the following 
	isomorphisms of $\pi$-formal group schemes
	$$
	\xymatrix{
		N^{[m]n}G \ar[d]^\phm \ar[r]^-\sim  &J^{n-1}(N^{[m]1}G)\ar[rr]^-{J^{n-1}\left(\Psi_1^{[m]}\right)} \ar[d]^{\phm} & & J^{n-1}(\hG^g) \ar[d]^{J^{n-1}(\pi)} \ar[r]^\sim & \WW_{n-1}^g \ar[d]^{\pi}
		\\
		N^{[m-1]n}G \ar[r]^-\sim& J^{n-1}(N^{[m-1]1}G)  \ar[rr]_-{J^{n-1}\left(\Psi_{1}^{[m-1]}\right)} & & J^{n-1}(\hG^g) \ar[r]^\sim & \WW_{n-1}^g. \\
	}
	$$
The horizontal maps $J^{n-1}\left(\Psi^{[m-1]}_1\right)$
	and $J^{n-1}\left(\Psi_1^{[m]}\right)$ are isomorphisms and therefore we have
	$$
	\xymatrix{
		N^{[m]n}G \ar[d]_\phm \ar[r]^-\sim & \WW_{n-1}^g \ar[d]^\pi \\
		N^{[m-1]n}G \ar[r]^-\sim & \WW_{n-1}^g. \\
	}$$
\end{theorem}

\begin{proof}
	Let $U \subset G$ be an affine open subscheme that contains the identity section
	over $S$ with an \'{e}tale map $f: U \map H = \Spf \hA^g$. By (\ref{Nmn_Witt})
	we have $N^{[m]n}G \simeq N^{[m]n}H$.
	
	Hence combining the above with equation (\ref{t0}) we obtain 
	$$
	\Psi_1^{[m-1]} \circ \phm = \pi \phm
	$$
	that is \cy{ the following diagram commutes
		$$\xymatrix{
			N^{[m]1}G \ar[d]_\phm \ar[r]^-{\Psi^{[m]}_1} & \hG^g \ar[d]^\pi\\
			N^{[m-1]1}G \ar[r]_-{\Psi^{[m-1]}_1} & \hG^g.
		}$$}
	We obtain our result by applying the functor $J^{n-1}$ to the above 
	diagram.
\end{proof}

For $m=0$, Borger and Saha in \cite{BoSa1} (Theorem 4.4) showed that the morphism $(\phi\circ\iota_0-\iota_0\circ \mff_0):N^{[0]n}G\rightarrow J^{n-1}G$ factors uniquely through the projection $u_0:N^{[0]m}G\rightarrow N^{[0]1}G$. 
Our next result generalizes this statement to arbitrary $m\geq 0$.

\begin{theorem}\label{N_[m]nG}
	Let $G$ be a smooth commutative $\pi$-formal group scheme over $S$ of relative dimension $g$. Then the morphism of $\pi$-formal group schemes $(\phi^{\circ (m+1)}\circ \iota_m-\phi^{\circ m}\circ \iota_m\circ \mf{f}_m):N^{[m]n}G\longrightarrow J^{n-1}G$ factors uniquely through $u:N^{[m]n}G\longrightarrow N^{[m]1}G$ as follows
	$$
	\xymatrix{
		N^{[m]n}G\ar[rrrr]^{(\phi^{\circ (m+1)}\circ \iota_m-\phi^{\circ m}\circ \iota_m\circ \mf{f}_m)}\ar[dd]_{u} & &&& J^{n-1}G\\
		& &  &\\
		N^{[m]1}G.\ar@{.>}[rrrruu]_{h_m} & & &&
	}
	$$
\end{theorem}
\begin{proof}
	Let $e:S\rightarrow G$ denote the identity section and we consider the following diagram
	$$
	\xymatrix{
		N^{[m]n}G\ar[rr]^{ w}\ar[d]^{x\mapsto(x,-x)}\ar@/_6.3pc/[ddddd]_{\phi^{\circ (m+1)}\circ \iota_m-\phi^{\circ m}\circ \iota_m\circ \mf{f}_m} & &  \prod_{[m]n}G\ar[d]_{z\mapsto(z,-z)}\ar@/^6.3pc/[ddddd]\\
		N^{[m]n}G\times N^{[m]n}G\ar[rr]^w \ar[d]^{(\iota_m,\mf{f}_m)} & & \prod_{[m]n}G\times \prod_{[m]n}G\ar[d]_{(\red{\iota_{m_w}},\mf{f}_{m_w})}\\
		J^{m+n}G\times N^{[m]n-1}G\ar[rr]^w \ar[d]^{(\phi^{\circ (m+1)},\iota_{m})} & & \prod_\phi^{m+n}G\times \prod_{[m]n-1}G\ar[d]_{(\phi_w^{\circ (m+1)},\red{\iota_{m_w}})}\\
		J^{n-1}G\times J^{n+m-1}G\ar[rr]^w \ar[d]^ {(id,\phi^{\circ m})} & & \prod_\phi^{n-1}G\times \prod_\phi^{m+n-1}G\ar[d]_{(id,\phi_w^{\circ m})}\\
		J^{n-1}G\times J^{n-1}G\ar[rr]^{w}\ar[d]^{(x,y)\mapsto x+y} & & \prod_\phi^{n-1}G\times\prod_\phi^{n-1}G\ar[d]_{(x,y)\mapsto x+y}\\
		J^{n-1}G\ar[rr]^w & &\prod_\phi^{n-1}G.
	}
	$$
	Let $z\in\prod_{[m]n}G$, then $z$ can be expressed as $z=\langle e,\cdots,e,z_{m+1},\cdots, z_{m+n}\rangle$, where $z_i\in G^{\phi^i}$. The composite of the maps in the 
right-hand column sends  $z\mapsto \langle \phi_S^{m+1}(z_{m+1}),e,\cdots,e\rangle.$
	Hence the map $\prod_{[m]n}G\rightarrow \prod_\phi^{n-1}G$ factors as $\prod_{[m]n}G\rightarrow\prod_{[m]1}G\rightarrow \prod_\phi^{n-1}G$ as 
	\begin{equation*}
		\langle e,\cdots,e,z_{m+1},\cdots, z_{m+n}\rangle\mapsto\langle e,\cdots, e,z_{m+1}\rangle\mapsto\langle\phi_S^{m+1}(z_{m+1}),e,\cdots,e\rangle.
	\end{equation*}
Thus, the solid arrows in the following diagram commute
	$$
	\xymatrix{
		& N^{[m]1}G\ar@{.>}[ldd]\ar[r]^w &\prod_{[m]1}G\ar[ldd]\\
		N^{[m]n}G\ar[r]^w\ar[d]_{\phi^{\circ (m+1)}\circ \iota_m-\phi^{\circ m}\circ \iota_m\circ \mf{f}_m}\ar[ur]^-{u} & \prod_{[m]n}G\ar[d]\ar[ru] & \\
		J^{n-1}G \ar[r]_w & \prod_\phi^{n-1}G. & 
	}
	$$
	We claim that there is a unique map $h_m:N^{[m]1}G\rightarrow J^{n-1}G$ making the whole diagram above commutative.


	Let $\bx_0=(x_{01},\cdots,x_{0g})$ be a system of \'etale coordinates at identity section $e$ of $G$.  From (\ref{Nmn_Witt}) we have $N^{[m]n}G\simeq \Spf R[\bx_{m+1},\cdots,\bx_{m+n}]^{\hat{}}$, where $(\bx_0,\cdots,\bx_{m+n})$ is the system of local Witt coordinates of length $m+n+1$ of $G$.
	Therefore we have a section of the projection $u$, say 
$\sigma_m:N^{[m]1}G\rightarrow N^{[m]n}G$ given in local Witt coordinates by $(\bx_{m+1})\mapsto (\bx_{m+1},e,\cdots,e)$ in terms of the local Witt coordinates. \\
	We put $h_m=(\phi^{\circ (m+1)}\circ \iota_m-\phi^{\circ m}\circ \iota_m\circ \mf{f}_m)\circ \sigma_m.$ Then to show commutativity, it is enough to prove $$\phi^{\circ (m+1)}\circ \iota_m-\phi^{\circ m}\circ \iota_m\circ \mf{f}_m=h_m\circ u.$$ 
	Since $G$ is smooth, $J^{m+n}G$ and $N^{[m]n}G$ are all smooth and in particular flat over $S$. Hence it is sufficient to consider a test ring $B$ which is a flat $\pi$-adically complete $R$-algebra and show that the ghost map $w:J^{n-1}G(B)\rightarrow \prod_{\phi}^{n-1}G(B)$ is injective. By adjointness, this is equivalent to showing $G(W_{n-1}(B))\rightarrow G(\prod_\phi^{n-1}B)$ to be injective and this, in turn, is equivalent to showing $w:W_{n-1}(B)\rightarrow\prod_\phi^{n-1}B$ is injective. But this holds since  $B$ is flat.
\end{proof}

\color{black}


\color{black}

\end{document}

%% file: newcommand.tex
\documentclass{amsart}
\usepackage{amssymb,amsmath,amsthm,epsf,epsfig,dsfont,bbm}
\usepackage{mathabx}
\usepackage{graphicx}

\usepackage{latexsym}
\usepackage{upref, eucal}
\usepackage[all]{xy}

\newcommand {\nc} {\newcommand}
\newcommand {\enm} {\ensuremath}

\nc {\bdm} {\begin{displaymath}}
\nc {\edm} {\end{displaymath}}

\newtheorem {theorem} {\bf{Theorem}}[section]
\newtheorem {lemma}[theorem] {\bf Lemma}
\newtheorem {proposition}[theorem] {\bf Proposition}

\newtheorem {corollary}[theorem] {\bf Corollary}
\numberwithin {equation}{section}

\newcommand\WW{\mathbb{W}}
\newcommand\ZZ{\mathbb{Z}}

     \newcommand\fp{\mathfrak{p}}

\newcommand{\Ou}{\enm{\mathcal{O}}}

\nc{\J}{\enm{\mathcal{J} }}
\nc {\Z} {\enm{\mathbb{Z}}}
\nc {\form}[1] {\enm{\mbox{\underline{for}}}_{#1}}
\nc {\prol}[1] {\enm{\mbox{\underline{prol}}_{{#1}^*}}}

\nc {\stk} {\stackrel}

\newcommand{\map}{\rightarrow}

\newcommand{\beqar}{\begin{eqnarray*}}
\newcommand{\eeqar}{\end{eqnarray*}}

\newcommand{\Pn}[2] {\ensuremath{ {\mathbb{P}}^{#1}_{#2}}}
\nc{\Quot}[3]{\enm{ {\mathfrak{Quot}_{ {#1}/{#2}/{#3}}}}}
\nc{\Hilb}[2]{\enm{ {\mathfrak{Hilb}_{ {#1}/{#2}}}}}
\newcommand{\mfrak}[1]{\mathfrak{#1}}
\newcommand{\mf}[1]{\mathfrak{#1}}
\newcommand{\bb}[1]{\mathbb{#1}}
\newcommand{\mcal}[1]{\mathcal{#1}}

\nc {\Coh}[4] {\ensuremath{H^{#1}(\Pn{#2}{},{#3}({#4}))}}
\nc {\Ch}[3] {\enm{H^{#1}(X_t,{#2}_t({#3}))}}
\nc {\Qphi}[4]{\enm{ {\mathfrak{Quot}^{~#4}_{ {#1}/{#2}/{#3}}}}}
\nc {\Gra}[4]{\enm{ {\mathfrak{Grass}_{#2}({#3},{#4})}}}
\nc {\HomA}[2]{\enm{\mathrm{Hom}_A{#1}{#2}}}
\nc {\tr}{\mathrm{tr}}

\nc {\C}[2]{\enm{\left(\begin{array}{l} {#1} \\ {#2} \end{array} \right)}}
\nc {\mat}[4]{\enm{\left(\begin{array}{ll}{#1} & {#2} \\ {#3} & {#4}
\end{array}\right)}}

\def \mb{\mbox}

 \def \Z{{\mathbb Z}}

   \def \h{\hat{\ }}

  \def \bX{{\bf X}} \def \bH{{\bf H}}

\def \hG{\hat{\mathbb{G}}_{\mathrm{a}}} 

\def \hA{\hat{\mathbb{A}}}

\def \R1{R((q))[q']\h}

\usepackage{xcolor}

\DeclareMathOperator{\Spf}{\mathrm{Spf}}

\newcommand{\Hom}{\mathrm{Hom}}

\newcommand{\oldmarginpar}[1]{}

\newcommand{\Em}{E_{[m]}}